\documentclass{proc-l}

\usepackage{graphicx}
\usepackage{url}
\usepackage{xcolor}

\newtheorem{thm}{Theorem}[section]
\newtheorem{lem}[thm]{Lemma}
\newtheorem{prop}[thm]{Proposition}

\theoremstyle{definition}
\newtheorem{defn}[thm]{Definition}

\theoremstyle{remark}
\newtheorem{remk}[thm]{Remark}

\numberwithin{equation}{section}

\newcommand{\Zcal}{ {\mathcal Z}}
\newcommand{\Rbb}{ {\mathbb R}}
\newcommand{\Zbb}{ {\mathbb Z}}

\newcommand{\crit}{{\rm crit}}
\newcommand{\supp}{{\rm supp}}

\begin{document}

\title{First mixed Laplace eigenfunctions with no hot spots}


\author{Lawford Hatcher}
\address{Department of Mathematics, Indiana University, Bloomington, IN, 47401, USA}

\subjclass[2020]{Primary }


\begin{abstract}
    The hot spots conjecture of J. Rauch states that the second Neumann eigenfunction of the Laplace operator on a bounded Lipschitz domain in $\Rbb^n$ attains its extrema only on the boundary of the domain. We present an analogous problem for domains with mixed Dirichlet-Neumann boundary conditions. We then solve this problem for Euclidean triangles and a class of planar domains bounded by the graphs of certain piecewise smooth functions. 
\end{abstract}

\maketitle

\section{Introduction}\label{intro}
Let $\Omega\subseteq \Rbb^2$ be a planar domain with piecewise smooth (i.e. $C^{\infty}$ away from at most finitely many points) boundary. Let $D,N\subseteq \partial \Omega$ denote non-empty relatively open subsets of the boundary such that $D\cap N=\emptyset$, $\overline{D\cup N}=\partial\Omega$, and $N$ is piecewise linear. We will study eigenfunctions corresponding to the lowest eigenvalue $\lambda_1$ of the following mixed Dirichlet-Neumann eigenvalue problem: 
\begin{equation}\label{efcneqn}
   \begin{cases}
    -\Delta u=\lambda u\;\;&\text{in}\;\;\Omega\\
    u=0\;\;&\text{in}\;\;D\\
    \partial_{\nu}u=0\;\;&\text{in}\;\;N,
    \end{cases} 
\end{equation}
where $\partial_{\nu}$ denotes the outward pointing normal derivative. Throughout the paper, we will let $u$ denote a first mixed eigenfunction for $(\Omega,D,N)$. By, for example, Courant's nodal domain theorem, $u$ does not change sign in $\Omega$, and it follows that $\lambda_1$ is a simple eigenvalue. Thus, $u$ is unique up to a scalar multiple, and we will assume throughout the paper that $u>0$ in $\Omega$.\\
\indent We will be particularly interested in the set of critical points of first mixed eigenfunctions. At smooth points of $\partial\Omega$ that are not endpoints of $\overline{D}$, $u$ extends to be infinitely differentiable, so we consider smooth boundary points to be potential critical points. However, we do not consider non-smooth boundary points to be critical points even when they are local extrema of $u$. We also do not consider endpoints of $\overline{D}$ to be critical points of $u$. Our main result describes the set of critical points of first mixed eigenfunctions on Euclidean triangles.

\begin{thm}\footnote{It was recently announced in \cite{yao1} that Li and Yao have a forthcoming proof that if $P$ is a triangle with $D$ an edge of $P$ such that the Neumann vertex is non-obtuse, then $u$ has no critical points.}
\label{mainthm}
    Let $P\subseteq \Rbb^2$ be a triangle, and let $D$ be either an edge or the union of two edges of $P$. Then each first mixed eigenfunction for $(P,D,N)$ has at most one critical point, and it is contained in $N$. Moreover, there exists a constant vector field $L$ such that $Lu>0$ in $P$.\footnote{We take the differential geometric viewpoint that vector fields act as first-order differential operators. That is, at each point $p\in P$, the vector field $L$ takes the partial derivative of $u$ in the direction $L(p)$.}
\end{thm}
\indent Theorem \ref{mainthm} is the combination of several results in Section \ref{triangletheorem}. The statements of these constituent results contain more precise information on the critical sets depending on the geometry of $P$ and whether $D$ equals one or two edges.\\
\indent Critical sets of eigenfunctions of the Laplace operator have been an active area of research for many years. In particular, J. Rauch conjectured in 1974 (see \cite{rauch}) that the first non-constant Neumann eigenfunction of a planar domain has local extrema only on the boundary. This is known as the \textit{hot spots conjecture}. Though the hot spots conjecture is still open in full generality, there are several partial results (see, e.g., \cite{burdzywerner}, \cite{banuelosburdzy}, \cite{miyamoto}, \cite{siudeja}, \cite{judgemondal}, \cite{erratum}, 
 and \cite{yao1}).\\
\indent Theorem \ref{mainthm} solves the analogous problem for the mixed Dirichlet-Neumann boundary conditions presented above. In this case, the first non-constant eigenfunction is the first mixed eigenfunction, making variational methods more powerful than in the case of the first non-constant Neumann eigenfunction (compare Lemma \ref{negativeintegral} below with Lemma 2.12 \cite{erratum}). However, the corresponding conjecture does not always hold in this setting, even in the case of simply connected planar domains (see Remark \ref{twosidedrectangle} for an explicit example). In fact, for any piecewise smooth $\Omega$, if $D$ is sufficiently large, then we expect that the geometry of the first mixed eigenfunction is similar to the geometry of the first Dirichlet eigenfunction, which necessarily has an interior extremum. When each first mixed eigenfunction for a triple $(\Omega,D,N)$ has no interior local extrema, we will say that $(\Omega,D,N)$ \textit{has no hot spots}.\\
\indent Though we are unable to determine all triples $(\Omega,D,N)$ having no hot spots even for polygonal domains, we prove in Theorem \ref{finitecritset} that, under suitable hypotheses, there are at most finitely many hot spots. We introduce a bit of terminology before stating the result. \\
\indent Let $\crit(u)$ denote the set of interior critical points of $u$. As stated above, $u$ extends to be infinitely differentiable at smooth points of $\partial\Omega$ that are not endpoints of $\overline{D}$. We will let $\overline{\crit}(u)$ be the set of critical points of this extension. We emphasize that non-smooth points of $\partial\Omega$ and endpoints of $\overline{D}$ are never considered to be elements of $\overline{\crit}(u)$. We will refer to $\overline{\crit}(u)$ as the \textit{critical set} of $u$. 
\begin{thm}\label{finitecritset}
    If $P$ is a simply connected polygon and $D$ is connected, then $\crit(u)$ is finite. If, in addition, $(P,D,N)$ does not consist of a rectangle with $D$ equal to a single edge, then $\overline{\crit}(u)$ is finite. 
\end{thm}

\begin{remk}\label{twosidedrectangle}
    We do not know whether $P$ being simply connected is a necessary hypothesis for Theorem \ref{finitecritset}. However the hypothesis that $D$ is connected in Theorem \ref{finitecritset} cannot be removed. For example, let $P$ be a rectangle with $D$ equal to the union of two opposite edges and $N$ equal to the union of the other two edges. The eigenfunctions in this case can be computed explicitly, and the critical set of each first eigenfunction is the line segment joining the midpoints of the Neumann edges of $P$. One wonders whether this is the only example of a triple $(P,D,N)$ with an infinite critical set. 
\end{remk}
We next define a large class of domains $\Omega$ and appropriate subsets $D,N\subseteq \partial\Omega$ for which $(\Omega,D,N)$ has no hot spots.

\begin{defn}\label{Gdom}
  Let $\Omega\subseteq\Rbb^2$ be a bounded Lipschitz domain. Suppose that, up to isometry, $\Omega$ is the region bounded by the $x$-axis and the graph of a continuous, piecewise smooth function $f:[a,b]\to[0,\infty)$ such that $f(x)>0$ for all $x\in(a,b)$. More precisely, $$\Omega=\{(x,y)\in\Rbb^2\mid a<x<b\;\;\text{and}\;\;0<y<f(x)\}.$$ We will call $\Omega$ a \textit{graph domain}.
\end{defn}

Not every piecewise smooth function yields a graph domain. For example, the domain bounded by the graph of $f(x)=x^2$ on $[0,1]$ is not a graph domain because it does not have Lipschitz boundary at $(0,0)$. However, any positive piecewise linear function gives a graph domain. See Figure \ref{Gdomains} for other examples.

\begin{figure}
    \centering
    \includegraphics[width=10cm]{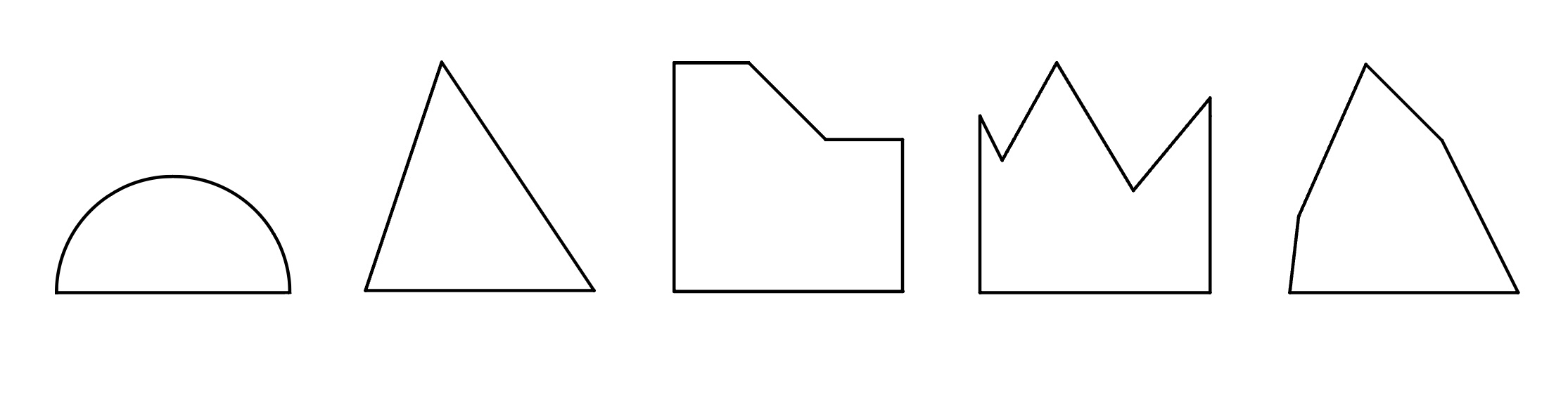}
    \caption{Some examples of graph domains. In Theorem \ref{Gpairs},
    $N$ equals the bottom edge of each of these domains and $D=\partial\Omega\setminus \overline{N}$.}
    \label{Gdomains}
\end{figure}

For some graph domain $\Omega$ bounded by a subset of the $x$-axis and bounded by the graph of a function $f:[a,b]\to [0,\infty)$, let $D=\partial\Omega\setminus([a,b]\times\{0\})$, and let $N=(a,b)\times\{0\}$. Let $u$ be a non-negative first mixed eigenfunction for $(\Omega,D,N)$. The next result shows that $(\Omega,D,N)$ has no hot spots and that we can bound the number of critical points and local extrema of $u$ with respect to the geometry of $\Omega$. Let $n$ be the number of strict local extrema $x$ of $f$ on $[a,b]$ for which $f(x)>0$ plus the number of intervals in $(f')^{-1}(0)$. For example, the value $n$ for each domain in Figure \ref{Gdomains}, from left to right, is $1$, $1$, $2$, $5$, and $1$. Note that $n$ may not be finite. 
\begin{thm}\label{Gpairs}\label{bound}
    Given a graph domain $\Omega$ with $D$ and $N$ as above, $-\partial_yu>0$ in $\Omega$.\footnote{Note that $\Omega$ is open. By the Neumann boundary conditions, $-\partial_yu$ vanishes on $N$.} In particular, $u$ has no interior critical points, and $\overline{\crit}(u)\subseteq N$. Moreover, $u$ has at most $n$ critical points on $N$. If $n$ is odd, then $u$ has at most $\frac{n+1}{2}$ local extrema. If $n$ is even, then $u$ has at most $\frac{n}{2}$ local extrema. If $n$ is infinite and $f(a),f(b)>0$, then $u$ has finitely many critical points on $N$.
\end{thm}

\begin{remk}\label{cabrechanillo}
    Let $\Omega$ be a graph domain with $D$ and $N$ as above. Let $\Omega'$ be the union of $\Omega$ 
    and its reflection over $N$. Let $u'$ be the extension of a first
    mixed eigenfunction $u$ of $(\Omega,D,N)$ to 
    $\Omega'$ via reflection. Since $u'\geq 0$, it follows that $u'$ is a
    first Dirichlet eigenfunction of $\Omega'$. If $\Omega'$ is smooth and strictly convex, the fact that there is exactly one critical point follows from the main result of \cite{CC}.
\end{remk}

Theorem \ref{Gpairs} is somewhat surprising because we expect there to exist hot spots when $D$ is sufficiently large. However, we can construct graph domains where $D$ comprises an arbitrarily large proportion of the boundary of $\Omega$ (for example, take $\Omega$ to be an acute triangle with $N$ equal to its shortest edge). \\
\indent The paper is organized as follows. We begin by studying the behavior of eigenfunctions near the vertices of polygons and graph domains in Section \ref{vertex}. We then use these results and the methods demonstrated in \cite{remarksoncritset} to prove Theorem \ref{finitecritset} in Section \ref{finite}. As in the Judge-Mondal proof of the hot spots conjecture for triangles, we then study in Section \ref{killing} the zero-level sets of various derivatives of first mixed eigenfunctions. In Section \ref{graphtheorem}, we prove Theorem \ref{bound}. In Section \ref{critedge}, we provide a technical classification of the critical points of first mixed eigenfunctions that will be essential in the proof of Theorem \ref{mainthm} in the case of obtuse triangles. Finally, in Section \ref{triangletheorem}, we prove Theorem \ref{mainthm}, which is the combination of Propositions \ref{obtusetwosided}, \ref{acuteNeumann}, and \ref{obtuseNeumann} below.

\section{Analysis at a vertex}\label{vertex} 
Let $\Omega$ be a polygon or graph domain. Let $v\in\partial\Omega$ be a boundary point such that, in a neighborhood of $v$, $\Omega$ is isometric to a circular sector given in polar coordinates by $S_{\epsilon}=\{re^{i\theta}\mid 0<\theta<\beta, 0<r<\epsilon\}$ for some $\beta\in(0,2\pi)$ and $\epsilon>0$. If such a sector exists, we will call $v$ a \textit{vertex} of $\Omega$. Let $u$ be an eigenfunction of the Laplace operator satisfying Neumann boundary conditions on $\partial\Omega\cap\partial S_{\epsilon}$. Using separation of variables, one can compute that $u$ has the following expansion valid in $S_{\epsilon}$ for some sufficiently small $\epsilon$ (see, e.g., \cite{judgemondal}):
\begin{equation}\label{neumannexpansion}
    u(re^{i\theta})=\sum_{n=0}^{\infty}a_n r^{n\nu}g_{n\nu}(r^2)\cos(n\nu\theta)
\end{equation}
where $a_n\in\Rbb$, $\nu=\frac{\pi}{\beta}$, and $r^{n\nu}g_{n\nu}(r^2)=J_{n\nu}(\sqrt{\lambda}r)$, where $J_{n\nu}$ is the Bessel function. Note that $g_{n\nu}$ is an entire function and $g^{(k)}_{n\nu}(0)\neq 0$ for all $k\in \Zbb_{\geq 0}$. \\
\indent Similarly, in a neighborhood of a Dirichlet vertex, we have the following expansion for $u$: 
\begin{equation}\label{dirichletexpansion}
u(re^{i\theta})=\sum_{n=1}^{\infty}b_nr^{n\nu}g_{n\nu}(r^2)\sin(n\nu\theta).
\end{equation}

Finally, suppose that $v$ is a mixed vertex whose Dirichlet edge is contained in the $x$-axis. Then we have the expansion
\begin{equation}\label{mixedexpansion}
    u(re^{i\theta})=\sum_{n=0}^{\infty}c_nr^{(n+\frac{1}{2})\nu}g_{(n+\frac{1}{2})\nu}(r^2)\sin\Big((n+\frac{1}{2})\nu\theta\Big).
\end{equation}

These expansions allow us to prove the following generalization of Lemma 18 of \cite{remarksoncritset}:

\begin{prop}\label{accumvertex}
    Let $u$ be a first mixed eigenfunction of $(\Omega,D,N)$. Then no vertex of $\Omega$ is an accumulation point of $\crit(u)$. If a vertex $v$ of $\Omega$ is an accumulation point of $\overline{\crit}(u)$, then $v$ is a Neumann vertex with angle $\pi/2$ or $3\pi/2$, and one of the edges adjacent to $v$ is a subset of $\overline{\crit}(u)$. 
\end{prop}
\begin{proof}
    First note that since $u$ is the first mixed eigenfunction, it is positive in $\Omega$ and vanishes only in $D$. Thus, $a_0$, $b_1$, and $c_0$ in the above expansions are all non-zero. In the case of Neumann vertices, the result then follows from Propositions 4.4 and 5.6 of \cite{judgemondal}.\\
    \indent In the case of a Dirichlet vertex, we have $$\partial_ru = b_1\nu r^{\nu-1}g_{\nu}(0)\sin(\nu\theta)+o(r^{\nu-1}).$$ For sufficiently small $r$ and $0<\theta<\beta$, therefore, we have $\partial_ru\neq 0$. For $\theta\in\{0,\beta\}$, the result follows from Lemma \ref{noDircritpoints} below. \\
    \indent The case of a mixed vertex is similar to the Dirichlet case. 
\end{proof}

\section{Finiteness of the critical set}\label{finite}
The main aim of this section is to prove Theorem \ref{finitecritset}. Let $u$ be a non-negative first mixed eigenfunction for $(\Omega,D,N)$ where $\Omega$ is a polygon or graph domain. 

\begin{lem}\label{noDircritpoints}
    $\overline{\crit}(u)\cap D=\emptyset$.
\end{lem}

\begin{proof}
    Since $u\geq 0$, $\Delta u=-\lambda u\leq 0$, so this follows from the Hopf lemma \cite{hopf}.
\end{proof}

For the next two results and the proof of Theorem \ref{finitecritset}, we suppose that $P$ is a polygon, $D$ is a union of edges of $P$, and $N=\partial P\setminus \overline{D}$.

\begin{prop}\label{interiorcrit}
    If $P$ is simply connected and $D$ is connected, then $\crit(u)$ does not contain an arc. 
\end{prop}
\begin{proof}
    Since $|\nabla u|^2$ is real-analytic away from the vertices, $\crit(u)$ is a locally finite graph with degree one vertices only in the boundary of $P$. Suppose that $\crit(u)$ does contain an arc $\gamma$. If this arc does not form a loop, then by Lemma \ref{noDircritpoints}, it has endpoints in $\overline{N}$. Thus, either $\gamma$ is a loop, or, since $P$ is simply connected and $D$ is connected, there is another curve $\eta\subseteq N$ such that $\overline{\gamma}\cup \eta$ is a loop where $\overline{\gamma}$ is the closure of $\gamma$. In either case, let $\Omega$ be the region bounded by this loop. The restriction of $u$ to $\Omega$ is then a non-constant, everywhere positive Neumann eigenfunction of $\Omega$. However, non-constant Neumann eigenfunctions must change signs, so we obtain a contradiction. 
\end{proof}

\begin{prop}\label{boundarycrit}
    If $P$ is simply connected, $D$ is connected, and $\overline{\crit}(u)\cap \partial P$ is infinite, then $P$ is a rectangle, and $D$ is an edge of $P$. 
\end{prop}
\begin{proof}
    By the real-analyticity of $|\nabla u|^2$ and Proposition \ref{accumvertex}, if $\overline{\crit}(u)\cap \partial P$ is infinite, then there is an edge $e\subseteq \overline{\crit}(u)$. By Lemma \ref{noDircritpoints}, $e\subseteq N$. Suppose that $e$ is contained in the $x$-axis and that, near this edge, $P$ lies in the upper half-plane. Then since the derivative $\partial_x$ commutes with the Laplacian, $\partial_xu$ is a Laplace eigenfunction that vanishes on $e$. Thus, $\partial_x\partial_xu=0$ on $e$. Since $u$ satisfies Neumann boundary conditions on $e$, we also have $\partial_y\partial_xu=\partial_x\partial_yu=0$. Since $\partial_xu$ vanishes on $e$ and $e\subseteq \overline{\crit}(\partial_xu)$, Lemma 2 of \cite{remarksoncritset} implies that $\partial_xu\equiv 0$ on $P$ (to apply the lemma, first extend $\partial_xu$ by reflection over $e$ to an open neighborhood of $e$).\\
    \indent Since $\partial_xu\equiv 0$, we see that $u$ is independent of $x$ and thus satisfies the ordinary differential equation $$\begin{cases}-\partial_y^2u=\lambda_1 u\;\;&\text{in}\;\;P\\\partial_yu=0\;\;&\text{in}\;\;\partial P\cap\{y=0\}\end{cases},$$ so $u$ is a scalar multiple of the function $\cos(\sqrt{\lambda_1}y)$. Since $u$ is positive off of $D$, it must be that $D$ is contained in the set $\{y=-\frac{\pi}{2\sqrt{\lambda_1}}\}\cup\{y=\frac{\pi}{2\sqrt{\lambda_1}}\}$, and $P$ is contained in the set $\{-\frac{\pi}{2\sqrt{\lambda_1}}<y<\frac{\pi}{2\sqrt{\lambda_1}}\}$. The Neumann edges of $P$ must then be vertical line segments. Since $D$ is connected, the result is thus proved. 
\end{proof}
\begin{proof}[Proof of Theorem \ref{finitecritset}]
    By computing the eigenfunction explicitly (see the proof of Proposition \ref{boundarycrit}), the result holds if $P$ is a rectangle and $D$ is an edge of $P$. Suppose therefore that $(P,D,N)$ is not of that form. If $\overline{\crit}(u)$ is infinite, then it contains an arc, contradicting Propositions \ref{interiorcrit} and \ref{boundarycrit}.
\end{proof}

We end the section by proving a finiteness result for graph domains that will be used in the proof of Theorem \ref{bound}.

\begin{prop}\label{finiteell}
    Let $\Omega$ be a graph domain with defining function $f:[a,b]\to[0,\infty)$. If $f(a)=0$ (resp. $f(b)=0$), then suppose that $f$ is linear in neighborhood of $a$ (resp. $b$). Let $u$ be a first mixed eigenfunction for $(\Omega,D,N)$. Then $\overline{\crit}(u)\cap N$ is a finite set. 
\end{prop}
\begin{proof}
    Suppose that $\overline{\crit}(u)\cap N$ is infinite. Let $p$ be an accumulation point of this set. By Proposition \ref{accumvertex}, $p$ is in the interior of $N$. The real-analyticity of $|\nabla u|^2$ implies that $N\subseteq \overline{\crit}(u)$. This contradicts Proposition \ref{accumvertex}.
\end{proof}

\section{Applying Killing fields to eigenfunctions}\label{killing}
We prove the other main theorems of the paper by studying the behavior of the zero-level sets of derivatives of each first mixed eigenfunction $u$. In particular, we will make use of derivatives given by constant vector fields parallel or perpendicular to certain edges of $\Omega$ as well as the rotational vector field $-y\partial_x+x\partial y$. In both cases, these vector fields commute with the Laplacian as in the proof of Proposition \ref{boundarycrit}. If $L$ denotes one of these vector fields, therefore, $Lu$ is also an eigenfunction of the Laplacian with the same eigenvalue as $u$ (though $Lu$ will not generally satisfy any specific boundary conditions). Before studying eigenfunctions of this form, we first prove a few results about general eigenfunctions with eigenvalue $\lambda_1$, where $\lambda_1$ is the first mixed eigenvalue of a polygon or graph domain $(\Omega,D,N)$. Throughout the section, we let $u$ be a corresponding first mixed eigenfunction. \\
\indent Given some function $\phi:\overline{\Omega}\to \Rbb$, let $\Zcal(\phi)=\phi^{-1}(\{0\})$. When $\phi=Xu$ with $X$ a constant or rotational vector field, then $\Zcal(\phi)$ is a locally finite graph (see Section 3 and Proposition 6.2 of \cite{judgemondal}) with degree one vertices only in the boundary of $\Omega$. The following results will primarily be applied to the restriction of $Xu$ to a connected component of $\Omega\setminus \Zcal(Xu)$. 
\begin{lem}\label{negativeintegral}
    Let $U\subseteq \Omega$ be an open set such that ${\rm int}(\Omega\setminus U)\neq \emptyset$. Let $\phi\in H^1(\Omega)\setminus \{0\}$ with $\supp \phi\subseteq U$ and $\phi\equiv 0$ on $D$. Further suppose that $\phi$ is smooth in $U$ and satisfies $-\Delta\phi=\lambda_1\phi$. Then $$\int_{\partial U\cap\partial \Omega}\phi\partial_{\nu}\phi>0.$$  
\end{lem}
\begin{proof}
    If not, then integration by parts gives $$\int_{\Omega}|\nabla \phi|^2=\lambda_1\int_{\Omega}|\phi|^2+\int_{\partial U\cap \partial \Omega}\phi\partial_{\nu}\phi\leq \lambda_1\int_{\Omega}|\phi|^2.$$ Using the variational formulation of the mixed eigenvalue problem, we see that we must actually have equality above and that $\phi$ is the first mixed eigenfunction of $(\Omega,D,N)$. However, since $\phi$ vanishes on an open set, this violates unique continuation, a contradiction. 
\end{proof}

We will use Lemma \ref{negativeintegral} along with the following formula to derive several contradictions in Section \ref{triangletheorem}. This formula was introduced by Terence Tao in \cite{polymath}, and another proof can be found in \cite{erratum}. In the following lemma, we let $P$ be a polygon, $D$ a union of edges of $P$, and $N=\partial P\setminus D$.
\begin{lem}\label{intFormula}
    Let $L$ be a constant vector field, and let $e\subseteq N$ be an edge of $P$. Suppose that $p$ (resp. $q$) is a point on $e$ that is either a critical point of $u$ or a Neumann vertex. Let $\ell$ be the line segment with endpoints $p$ and $q$. If $\partial_{\tau}$ is the unit length counterclockwise tangent vector to $e$ and $q-p$ points in the same direction as $\partial_{\tau}$, then $$\int_{\ell}Lu\partial_{\nu}Lu=-\frac{1}{2}\lambda_1\langle L,\partial_{\tau}\rangle \langle L,\partial_{\nu}\rangle\big(u(q)^2-u(p)^2\big).$$ 
\end{lem}

\begin{lem}\label{noloops}
    Let $U\subseteq \Omega$ be an open set. Suppose that $\phi\in H^1(U)\setminus\{0\}$ with $\phi|_{\partial U}\equiv 0$ and that $\phi$ is smooth in $U$. If $-\Delta\phi=\lambda \phi$, then $\lambda>\lambda_1$. 
\end{lem}
\begin{proof}
    Suppose that $\lambda\leq \lambda_1$. Extend $\phi$ to be equal to $0$ outside of $U$ so that $\phi\in H^1_0(\Omega)$. Using the variational formulation of the eigenvalue problem as in the proof of Lemma \ref{negativeintegral}, we see that $\phi$ is the first mixed eigenfunction for $(\Omega,D,N)$, which is absurd since it is identically $0$ on $\partial \Omega$.
\end{proof}

\begin{lem}\label{neumannloop}
    Suppose that $U\subseteq \Omega$ is an open set such that ${\rm int}(\Omega\setminus U)\neq \emptyset$ and that $\phi\in H^1(U)\setminus\{0\}$ is smooth in $U$. Further suppose that $\phi$ satisfies Neumann boundary conditions on $\partial U\cap N$ and that $\phi\equiv 0$ on $\partial U\setminus N$. If $-\Delta \phi=\lambda\phi$, then $\lambda>\lambda_1$.
\end{lem}
\begin{proof}
    Extend $\phi$ to be equal to $0$ outside of $U$. As in the proof of Lemma \ref{noloops}, if $\lambda\leq \lambda_1$, then the variational formulation of the eigenvalue problem shows that $\phi$ is the first mixed eigenfunction, contradicting unique continuation since $\phi$ vanishes on an open set.
\end{proof}
\begin{remk}
    Most of our applications of the above results will be applied to derivatives of a first mixed eigenfunction $u$. It is well known that eigenfunctions satisfying Dirichlet or Neumann conditions on a line segment extend to be analytic on (the interior of) the line segment. Near convex Neumann and Dirichlet vertices, expansions (\ref{neumannexpansion}) and (\ref{dirichletexpansion}) show that eigenfunctions are locally in $H^2$. Near mixed vertices with angle at most $\pi/2$, expansion (\ref{mixedexpansion}) shows that eigenfunctions are locally in $H^2$. Thus, near these vertices, derivatives of eigenfunctions are locally in $H^1$, and we can apply the above results to these derivatives near appropriate vertices. 
\end{remk}

Suppose that $\Omega$ contains a vertex $v$ such that $\Omega$ is isometric to a circular sector in a neighborhood of $v$. Using the expansions introduced in Section \ref{vertex}, we can determine exactly when $v$ is also a vertex of $\Zcal(Xu)$ when $X$ is a constant vector field. This is a generalization of Lemma 2.2 \cite{erratum} and Lemma 2.1 \cite{polygons}. In the results below, we always assume that the vertex is embedded in $\Rbb^2$ as described in Section \ref{vertex}. 
\begin{lem}\label{deg1Neumann}\label{deg1Dirichlet}\label{deg1Mixed}
    Let $v$ be a Neumann, Dirichlet, or mixed vertex of $\Omega$ with angle $\beta<2\pi$. Let $X$ be a constant vector field whose angle with the positive $x$-axis modulo $\pi$ is $\delta$. Then $v$ is at most a degree one vertex of $\Zcal(Xu)$. If $v$ is a Neuman vertex, then we have
    \begin{enumerate}
        \item If $\beta<\pi/2$ or $a_1=0$ and $\beta<\pi$, then $v$ is a vertex of $\Zcal(Xu)$ if and only if $\delta\in\big[\frac{\pi}{2},\frac{\pi}{2}+\beta\big]+\pi\Zbb.$
        \item If $\frac{\pi}{2}<\beta<\pi$ and $a_1\neq 0$, then $v$ is a vertex of $\Zcal(Xu)$ if and only if $\delta\in \big[\beta-\frac{\pi}{2},\frac{\pi}{2}\big]+\pi\Zbb.$
    \end{enumerate}
    If $v$ is a Dirichlet vertex, then we have 
    \begin{enumerate}
        \item If $0<\beta<\pi$, then $v$ is a vertex of $\Zcal(Xu)$ if and only if $\delta\in \big[\beta,\pi\big]+\pi\Zbb.$
        \item If $\pi<\beta<2\pi$, then $v$ is a degree one vertex of $\Zcal(Xu)$ if and only if $\delta\in\big[0,\beta-\pi\big]+\pi\Zbb.$
    \end{enumerate}
    Finally, if $v$ is a mixed vertex, then we have
    \begin{enumerate}
        \item If $0<\beta\leq\pi/2$, then $v$ is a vertex of $\Zcal(Xu)$ if and only if 
    $\delta\in\big[\beta+\frac{\pi}{2},\pi\big]+\pi\Zbb.$
        \item If $\pi/2<\beta<\pi$, then $v$ is a degree one vertex of $\Zcal(Xu)$ if and only if $\delta\in\big[0,\beta-\frac{\pi}{2}\big]+\pi\Zbb$
    \end{enumerate}
\end{lem}
\begin{proof}
    Because $u$ is a first mixed eigenfunction, it does not vanish in $\Omega$. Therefore, regardless of the boundary conditions, the first coefficient ($a_0$, $b_1$, or $c_0$) in the appropriate expansion never vanishes. The Neumann case then follows from Lemma 2.1 of \cite{polygons}. The Dirichlet and mixed cases are proved similarly. 
\end{proof}

\begin{lem}\label{deg1rotational}
    Let $v$ be a mixed vertex of $\Omega$ such that the Dirichlet edge adjacent to $v$ lies in the positive $x$-axis. Let $R$ be a rotational vector field centered at a point in the $x$-axis. If the angle $\beta$ at $v$ is less than $\pi$ and the center of $R$ is not the origin, then there is some neighborhood $U$ of $v$ such that $\Zcal(Ru)$ does not intersect $U$. If $R$ is centered at the origin, then there is a neighborhood $U$ of $v$ such that $\Zcal(Ru)\cap U$ is contained in the Neumann edge adjacent to $v$.
\end{lem}
\begin{proof}
    If $R$ is centered at the origin, then $R=\partial_{\theta}$, and one can use expansion (\ref{mixedexpansion}) to get the result since $c_0\neq 0$. A rotational vector field $R$ centered at a point $(a,0)\neq(0,0)$ can be expressed as $R=-y\partial_x+(x-a)\partial_y$. In polar coordinates, this is $R=-a\sin\theta\partial_r+(1-\frac{a}{r}\cos\theta)\partial_{\theta}$. Since $c_0\neq 0$, we have $$u(re^{i\theta})=c_0g_{\frac{1}{2}\nu}(0)r^{\frac{1}{2}\nu}\sin\Big(\frac{1}{2}\nu\theta\Big)+O(r^{\min\{\frac{1}{2}\nu+2,\frac{3}{2}\nu\}}).$$ We then compute $$Ru(re^{i\theta})=\frac{-a}{2}\nu c_0g_{\frac{1}{2}\nu}(0)r^{\frac{1}{2}\nu-1}\cos\Big(\big(\frac{1}{2}\nu-1\big)\theta\Big)+O(r^{\frac{1}{2}\nu}),$$ and the result follows. 
\end{proof}

\section{First mixed eigenfunctions of graph domains: proofs of Theorems \ref{Gpairs} and \ref{bound}}\label{graphtheorem}

Here we prove the Theorem \ref{bound}. We begin with a result that reduces the proof to the case that the function $f$ defining a graph domain is piecewise linear.

\begin{prop}\label{domainapprox}
    Let $\Omega_1\subseteq\Omega_2\subseteq...$ be a sequence of graph domains with $D_n$ and $N_n$ as in the statement of Theorem \ref{bound}. For each $n$, let $u_n$ be a first mixed eigenfunction for $(\Omega_n,D_n,N_n)$ with $\|u_n\|_{L^2(\Omega_n)}=1/\sqrt{2}$. Suppose that $\Omega=\cup_n\Omega_n$ is also a graph domain. Identify each $u_n$ with its extension by $0$ to a function on $\Omega$. Then, passing to a subsequence if necessary, $\{u_n\}$ converges in $H^1(\Omega)$ to a first mixed eigenfunction $u\neq 0$ of $(\Omega,D,N)$.
\end{prop}
\begin{proof}
    \indent As in Remark \ref{cabrechanillo}, extend each $u_n$ by reflection to the first Dirichlet eigenfunction $u_n'$ of the double $\Omega_n'$ of $\Omega_n$. Then $\|u_n'\|_{L^2(\Omega')}=1$ (where $\Omega'$ is the double of $\Omega$) for all $n$. We will show that some subsequence of $\{u_n'\}$ converges in $H^1_0(\Omega')$ to a non-negative first Dirichlet eigenfunction $u'$ of $\Omega'$. This eigenfunction restricts to a non-negative first mixed eigenfunction $u$ for $(\Omega,D,N)$.\\
    \indent For each $n$, let $\lambda_1^n$ denote the first mixed eigenvalue of $(\Omega_n,D_n,N_n)$. Because $\Omega'$ has Lipschitz boundary, Theorem 1.2.2.2 of \cite{grisvard} shows that $\Omega'$ has the restricted cone property of Definition 2.1 of \cite{agmon}. We may thus apply Theorem 1.5 of \cite{rauchtaylor} to see that $\lambda_1^n\to\lambda_1$ as $n\to\infty$. Since $\|u_n'\|_{H^1_0(\Omega')}^2=\lambda_1^n$ is uniformly bounded in $n$, we may pass to a subsequence, also denoted $\{u_n'\}$, that converges weakly in $H^1_0(\Omega')$ to some $u'$. Since the $u_n'$ are $L^2$-normalized, we have $\|u'\|_{L^2(\Omega')}=1$. For all $\phi\in C^{\infty}_0(\Omega')$, the support of $\phi$ is compactly contained in the support of $u_n'$ for all $n$ sufficiently large, so we have $$\int_{\Omega'}\nabla u'\cdot\nabla \phi=\lim_{n\to\infty}\int_{\Omega'}\nabla u_n'\cdot\nabla\phi=\lim_{n\to\infty}\lambda_1^n\int_{\Omega'}u_n'\phi=\lambda_1\int_{\Omega'}u'\phi.$$ It follows that $u'$ is a non-negative first Dirichlet eigenfunction for $\Omega'$. Therefore, $$\int_{\Omega'}|\nabla (u_n'-u')|^2=\int_{\Omega'}|\nabla u_n'|^2+\int_{\Omega'}|\nabla u'|^2-2\int_{\Omega'}\nabla u_n'\cdot \nabla u=\lambda_1^n+\lambda_1-2\lambda_1^n\int_{\Omega'}u_n'u'\to0.$$
\end{proof}

    \begin{proof}[Proof of Theorem \ref{bound}]
    \indent The first two sentences follow from Theorem 1.3 of \cite{nirenberg}. We will prove the remaining statements first for graph domains with piecewise linear defining functions. By approximating a general $f$ by a non-decreasing sequence of piecewise linear functions, each of which having at most as many local extrema and stationary intervals as $f$, the general result for $n$ finite follows from Proposition \ref{domainapprox}. If $n$ is infinite and $f(a),f(b)>0$, then Proposition \ref{finiteell} shows that the critical set is finite. \\
    \indent So suppose that $f$ is piecewise linear. By Proposition \ref{finiteell}, $u$ has at most finitely many critical points on $N$. Since $u$ vanishes at the endpoints of $N$ and $u\geq 0$, the number $\ell$ of critical points that are local extrema of $u|_{N}$ is odd. By extending $u$ to a neighborhood of $N$ via reflection, $u$ is a non-constant subharmonic function in a neighborhood of $N$, so it cannot have any local minima by the strong maximum principle for subharmonic functions. Thus, the local minima of $u|_{N}$ are not local extrema of $u$. The extrema of $u|_{N}$ nearest the endpoints of $N$ must be local maxima since $u$ is positive and vanishes at the endpoints of $N$. Thus, $u$ has at most $\frac{\ell+1}{2}$ local extrema on $N$. Suppose that $u$ has $k$ critical points on $N$, so $\ell\leq k$. Let $n$ be as in the statement of the theorem. We will show that $k\leq n$. It will follow that $\ell\leq n$ when $n$ is odd, and $\ell\leq n-1$ when $n$ is even. Note that by the piecewise linearity assumption, $n$ is finite.\\
    \indent Let $X=\partial_x$ be the constant vector field tangent to $N$. By Lemma 6.6 of \cite{judgemondal}, each critical point of $u$ in $N$ is an endpoint of an arc in $\Zcal(Xu)$ that intersects $\Omega$, and by Lemma \ref{neumannloop}, the other endpoint of the arc emanating from each of these points cannot lie in $N$. By Lemmas \ref{noDircritpoints} and \ref{deg1Dirichlet}, the other endpoint of each of these arcs must be a point $(x,y)\in D$ such that either $x$ is a strict local extremum for $f$ or $x$ lies in an interval on which $f$ is constant. By Lemmas \ref{neumannloop} and \ref{deg1Dirichlet}, at most one of these arcs can terminate at each of these points. Suppose that $f'(x)\equiv 0$ on some open interval $I$, so $Xu\equiv 0$ on $I\times f(I)$. At an endpoint $x$ of $I$, Lemma \ref{deg1Dirichlet} shows that $I\times f(I)$ is the only arc in $\Zcal(Xu)$ terminating at $(x,f(x))$. Lemma \ref{neumannloop} shows that at most one arc in $\Zcal(Xu)$ with an endpoint in $N$ has an endpoint in $I\times f(I)$, so $n$ dominates the number of arcs of $\Zcal(Xu)$ with endpoints in $N$, and $k\leq n$.
\end{proof}

\section{Critical points on an edge}\label{critedge}
To prove Proposition \ref{obtuseNeumann} below, we will classify the critical points of $u$ by their so-called indices. Judge and Mondal used this classification in \cite{erratum} to prove the hot spots conjecture for acute triangles. Here we recall several of the facts from \cite{erratum} regarding the indices of critical points. Throughout this section, $P$ is a triangle, $D$ is the union of some collection of edges of $P$, $N$ is the union of the other edges, and $u$ is a first mixed eigenfunction for $(P,D,N)$. 
\begin{defn}
    Extend $u$ to a neighborhood of the Neumann edges of $P$ via reflection. Let $p\in\overline{P}$ be a critical point of $u$. Suppose that, near $p$, the point $p$ is a degree $n$ vertex of $u^{-1}(\{u(p)\})$. We then say that $p$ has \textit{index} $1-\frac{n}{2}$.
\end{defn}
For example, a local extremum of $u$ that is not a vertex of $P$ is an index $1$ critical point. 
\begin{lem}\label{indexpom1}
    The index of a critical point of $u$ is either $-1$, $0$, or $1$. 
\end{lem}
\begin{proof}
    This follows from the fact that $u$ does not vanish in $\overline{P}\setminus \overline{D}$. See Proposition 2.5 of \cite{erratum} for details. 
\end{proof}

\begin{lem}\label{notextremum}
    Let $e$ be a Neumann edge of $P$. Let $p\in e$ be a critical point of $u$ that is not a local extremum of $u|_{e}$. Then $p$ has index $0$. 
\end{lem}
\begin{proof}
    Since $p$ is not a local extremum of $u|_{e}$, it is certainly not an extremum of $u$ and thus does not have index $1$. Since $u$ is non-constant on $e$ (see Theorem \ref{finitecritset}), there is a neighborhood $U$ of $p$ such that $U\cap u^{-1}(u(p))\cap e=\{p\}$. Since $u$ is not a local extremum of $u|_{e}$, it follows that $u$ does not have index $-1$ using the symmetry of the level set about $e$. By Lemma \ref{indexpom1}, the index of $p$ is therefore equal to $0$.
\end{proof}

\begin{prop}\label{indexzero}
    Let $X$ be a constant vector field. Let $e$ be an edge of $P$. If $p\in e$ is a critical point of $u$ that is not a local extremum of $u|_{e}$, then $p$ is not a degree one vertex of $\Zcal(Xu)$.
\end{prop}
\begin{proof}
    Since $p$ is not a local extremum of $u|_{e}$, Lemma \ref{notextremum} implies that $p$ has index $0$. The statement is then the contrapositive of Proposition 2.7 of \cite{erratum}.
\end{proof}

\begin{remk}\label{atleasttwo}
    Proposition \ref{indexzero} does not rule out $p$ being a vertex of degree greater than one of $\Zcal(Xu)$. In fact, if $X$ is parallel to the edge $e$, then Lemma 6.6 \cite{judgemondal} shows that $p$ is a degree at least two vertex of $\Zcal(Xu)$.
\end{remk}

\section{First mixed eigenfunctions on triangles: proof of Theorem \ref{mainthm}}\label{triangletheorem}

We now prove three propositions that constitute Theorem \ref{mainthm}.
\begin{prop}\label{obtusetwosided}\label{acutetwosided}
    Let $P$ be a triangle and $N$ be an edge of $P$. Let $D$ be the union of the other two edges. Then each first mixed eigenfunction $u$ for $(P,D,N)$ has exactly one critical point. This critical point lies in $N$ and is the global maximum of $u$. Moreover, if $L$ is a constant vector field that bisects the angle of the Dirichlet vertex, then $Lu>0$ in $P$. 
\end{prop}
\begin{proof}
    By Lemma \ref{deg1Dirichlet}, the Dirichlet vertex is not an endpoint of any arc in $\Zcal(Lu)$. By Lemma \ref{noDircritpoints}, no arc of $\Zcal(Lu)$ has an endpoint in the interior of one of the Dirichlet edges. Hence, if $\Zcal(Lu)$ intersects $P$, then by Lemma \ref{noloops}, some arc in $\Zcal(Lu)$ has two distinct endpoints in $N$, and these endpoints are critical points of $u$. If $L$ is orthogonal to the Neumann edge, then $Lu$ vanishes on $N$, and Lemma \ref{noloops} shows that $Lu$ does not vanish in $P$. Let $X$ be the constant vector field parallel to $N$. Then the Dirichlet vertex is a degree one vertex of $\Zcal(Xu)$, and $Xu$ does not vanish in the interior of either Dirichlet edge by Lemma \ref{noDircritpoints}. Each critical point on $N$ is an endpoint of an arc in $\Zcal(Xu)$ that intersects $P$. These arcs cannot intersect by Lemma \ref{neumannloop}, and at most one of them can terminate at the Dirichlet vertex. It follows that there is at most one critical point on $N$, so $Lu\neq 0$ in $P$. Since $u\geq 0$, it follows that $Lu>0$ in $P$, and $u$ has exactly one critical point on $N$. 
\end{proof}

The next result proves the case in Theorem \ref{mainthm} where $D$ is one edge of $P$ and where the Neumann vertex is non-obtuse.
\begin{prop}\label{acuteNeumann}
    Suppose that $P$ is a triangle and that $D$ is an edge of $P$. Let $N=\partial P\setminus \overline{D}$. If the Neumann vertex $v$ of $P$ has angle $\leq\frac{\pi}{2}$, then $\overline{\crit}(u)$ is empty. In particular, $v$ is the unique local (and hence the global) maximum of $u$. Moreover, if $L$ denotes the constant vector field that restricts to the inward normal vector field to $D$, then $Lu>0$ in $P$.
\end{prop}

In the case where the Neumann vertex has a right angle, we will see that this result is a corollary of Theorem \ref{Gpairs}. We begin with two preparatory lemmata. 

\begin{lem}\label{localmax}
    Let $P$ be a polygon and $D$ some collection of edges of $P$. Each Neumann vertex with angle less than $\frac{\pi}{2}$ is a local maximum of each non-negative first mixed eigenfunction $u$ of $(P,D,N)$. 
\end{lem}
\begin{proof}
    This follows from the expansion (\ref{neumannexpansion}) and the fact that $u>0$ in $\overline{P}\setminus \overline{D}$. See Proposition 2.1 of \cite{erratum} for details.  
\end{proof}

\begin{lem}\label{interiorimpliesexterior}
    Let $P$ be a triangle, and let $D$ be an edge of $P$. Let $N=\partial P\setminus \overline{D}$. If the first mixed eigenfunction $u$ for $(P,D,N)$ has an interior critical point, then $u$ has at least one critical point on each Neumann edge.
\end{lem}
\begin{proof}
    Let $v$ denote a mixed vertex of $P$ and $R_v$ the rotational vector field centered at $v$. Then $R_vu$ vanishes on the Neumann edge adjacent to $v$, and $R_vu$ cannot vanish in the interior of $D$ by Lemma \ref{noDircritpoints}. By Lemma \ref{deg1rotational}, $\Zcal(R_vu)$ does not have a vertex at the opposite endpoint of $D$. If $u$ has an interior critical point, then $\Zcal(R_vu)$ has an arc that intersects $P$, and by Lemma \ref{noloops}, this arc must have distinct endpoints in $\partial P$. By Lemma \ref{noloops}, these endpoints cannot both be in the Neumann edge adjacent to $v$. Thus, $\Zcal(R_vu)$ must have an endpoint in the interior of the edge $e$ opposite to $v$. Since $R_vu$ is nowhere orthogonal to the edge opposite to $v$ and this edge is a Neumann edge, this endpoint is a critical point of $u$. A symmetric argument yields a critical point on the other Neumann edge. 
\end{proof}

\begin{proof}[Proof of Proposition \ref{acuteNeumann}]
    Let $e'$ and $e''$ be the Neumann edges of $P$, and let $v$ be the Neumann vertex. We first prove the theorem in the case that the angle at $v$ is strictly less than $\frac{\pi}{2}$. Let $L_{e'}$ and $L_{e''}$ be the constant vector fields tangent to $e'$ and $e''$, respectively. Let $n'$ (resp. $n''$) denote the number of critical points on $e'$ (resp. $e''$). By Theorem \ref{finitecritset}, $n'$ and $n''$ are finite. By Lemma \ref{interiorimpliesexterior}, to show that $u$ has no interior critical points, it suffices to show that $n'=n''=0$. Suppose toward a contradiction that there exists a critical point $p$ on $e'$, so $n'>0$. Then $p$ is an endpoint of an arc in $\Zcal(L_{e'}u)$ that, by Lemma \ref{neumannloop}, must have another endpoint in $e''$ that is not $v$. Let $v'$ be the (mixed) vertex opposite to $e'$. By Lemma \ref{deg1Mixed}, this vertex is not a vertex of $\Zcal(L_{e'}u)$, so the other endpoint to this arc must be in the interior of $e''$, and this endpoint is a critical point of $u$. Arcs of $\Zcal(L_{e'}u)$ emanating from distinct points in $e'$ cannot intersect each other by Lemma \ref{neumannloop}. Thus, $n'\leq n''$. By a symmetric argument, $n''\leq n'$, so $n'=n''$. Let $p_{e'}$ be the nearest critical point on $e'$ to $v$, and let $p_{e''}$ be the nearest critical point on $e''$ to $v$. By the argument above, $\Zcal(L_{e''}u)$ contains an arc joining $p_{e'}$ to $p_{e''}$. Let $\Omega$ be the region bounded by this arc whose closure contains $v$ (see Figure \ref{acuteomega}). By Lemma \ref{localmax}, $v$ is a local maximum, so $u(p_{e'})<u(v)$. Lemma \ref{intFormula} gives
    $$\int_{\partial\Omega\cap\partial P}L_{e''}u\partial_{\nu}L_{e''}u<0$$ since the angle at $v$ is less than $\frac{\pi}{2}$, contradicting Lemma \ref{negativeintegral}.\\
        \begin{figure}
        \centering
        \includegraphics[width=5cm]{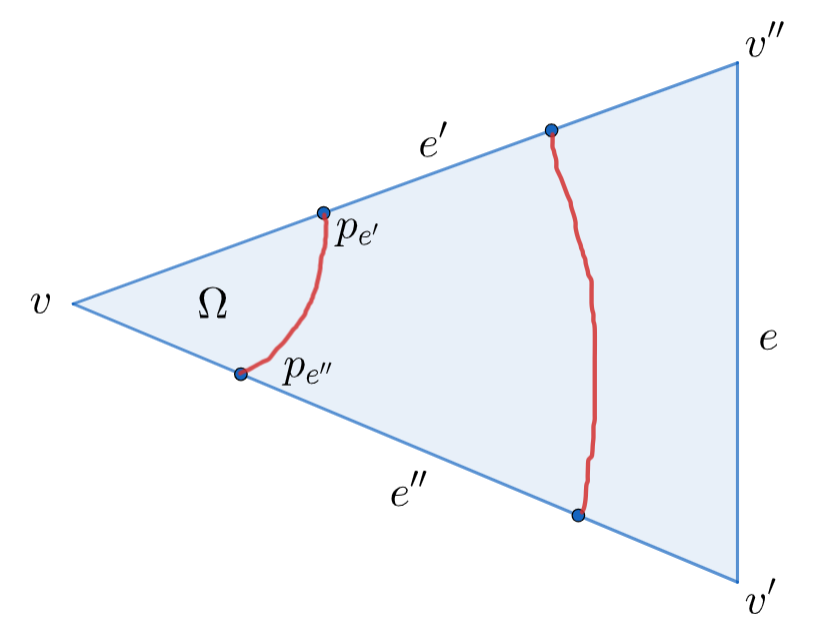}
        \caption{Illustration of the situation in the proof of Theorem \ref{acuteNeumann}. The red arcs are $\Zcal(L_{e''}u)$.}
        \label{acuteomega}
    \end{figure}
    \indent Now suppose that the angle at $v$ is $\frac{\pi}{2}$. By reflecting $P$ over one of its Neumann edges, we obtain another triangle $P'$, and we can extend $u$ to a function $u'$ on $P'$ via reflection. Since $u'>0$ in $P'$ and $u'\equiv0$ on two edges $e_1$ and $e_2$ of $P'$, $u'$ is the first mixed eigenfunction of $(P',e_1\cup e_2,\partial P'\setminus\overline{e_1\cup e_2})$. Since $P$ is a right triangle, $P'$ is a graph domain, and we may apply Theorem \ref{Gpairs} to see that $u'$ has a unique local extremum on the Neumann edge of $P'$. Since $u'$ is even about the line of symmetry of $P'$, this extremum must be at the Neumann vertex of $P$.\\
    \indent For either of the above cases, let $V$ be the constant vector field that restricts to the inward unit normal on $D$. Then $Vu$ cannot vanish in $D$ and cannot have a degree one vertex in either of the Neumann edges by the above. By Lemma \ref{deg1Mixed}, the endpoints of $D$ are also not vertices of $\Zcal(Vu)$. At most one arc of $\Zcal(Vu)$ can have an endpoint at the Neumann vertex. Since this arc cannot form a loop by Lemma \ref{noloops}, it therefore cannot exist, so $Vu$ does not vanish in $P$. Since $u$ was chosen to be positive, it follows that $Vu>0$ in $P$.
\end{proof}

The last case in Theorem \ref{mainthm} is 
\begin{prop}\label{obtuseNeumann}
    Suppose that $P$ is a triangle and that $D$ is an edge of $P$ such that the Neumann vertex $v$ of $P$ has angle $>\frac{\pi}{2}$. If $P$ is isosceles, then $\overline{\crit}(u)=\emptyset$, and $v$ is the unique local (and hence global) maximum of $u$. If $P$ is not isoceles, then $\overline{\crit}(u)$ consists of a single point that is contained in the longer Neumann edge. This critical point is the unique local (and hence global) maximum of $u$. Moreover, in either case, if $L$ is the constant vector field extending the outward normal vector field to the longer Neumann edge, then $Lu> 0$ in $P$.
\end{prop}

\begin{proof}
     Let $v$ be the Neumann vertex of $P$. If $P$ is isosceles, then $u$ is even about the line segment bisecting the angle at $v$. Let $P'$ be one of the two triangles into which this angle bisector divides $P$. The restriction of $u$ to $P'$ is a first mixed eigenfunction, so $u$ having no critical points follows from Proposition \ref{acuteNeumann}.\\
     \indent Suppose that $P$ is not isosceles, and suppose toward a contradiction that $v$ is a local maximum of $u$. Then $a_1=0$ in expansion (\ref{neumannexpansion}). Let $e'$ and $e''$ be the Neumann edges of $P$, and let $e$ be the Dirichlet edge. Let $v'$ and $v''$ be the vertices of $P$ opposite to $e'$ and $e''$, respectively. Suppose without loss of generality that $e''$ is strictly shorter than $e'$. Let $v_1$ be the endpoint in $e$ of the line segment bisecting the angle at $v$. Let $v_2\in e'$ be the image of $v'$ under the reflection about this line segment. Suppose that $P$ is embedded in $\Rbb^2$ such that this angle bisector lies in the $x$-axis. On the kite $K$ with vertices $v$, $v'$, $v_1$, and $v_2$, define a function $$w(x,y)=u(x,y)-u(x,-y).$$ We will show that $w\equiv0$, which will imply that $u$ vanishes on the line segment joining $v_1$ to $v_2$, contradicting that $P$ is not isosceles and that $u>0$ in $P$.\\
    \indent Suppose to the contrary that $w$ is not identically $0$. Since $u>0$ in $P$, we have that $w>0$ on the interior of the line segment joining $v_1$ to $v_2$ and that $w<0$ on the interior of the line segment joining $v'$ to $v_1$. By applying a Euclidean isometry, suppose now that $v$ lies at the origin with one of its adjacent edges contained in the positive $x$-axis. Expansion (\ref{neumannexpansion}) then yields $$w(re^{i\theta})=\sum_{n\geq 3:\;n\;\text{odd}}2a_nr^{n\nu}g_{n\nu}(r^2)\cos(n\nu\theta),$$ where we used that $a_1=0$. By this expansion, at least $3$ arcs in $\Zcal(w)$ emanate from $v$. One of these arcs coincides with the angle bisector at $v$ since $w$ is odd about this line segment. Since $-\Delta w=\lambda_1 w$, Lemma \ref{noloops} shows that these arcs cannot intersect each other or themselves anywhere except at $v$. Since $w\neq 0$ on the interiors of the two edges adjacent to $v_1$, two of these arcs must have an endpoint either in the edge joining $v$ to $v'$ or the edge joining $v$ to $v_2$. However, since $w$ satisfies Neumann boundary conditions on these edges, this contradicts Lemma \ref{neumannloop}. We therefore have that $v$ is not a local maximum of $u$.\\
    \indent By Lemma \ref{interiorimpliesexterior}, to show that $u$ has no interior critical points, it suffices to show that there are not critical points on both Neumann edges. Since $v$ is not an extremum of $u$, $a_1\neq0$ in the expansion  (\ref{neumannexpansion}) for $v$. Let $n'$ (resp. $n''$) be the number of critical points on $e'$ (resp. $e''$). Let $s'$ (resp. $s''$) denote the number of critical points on $e'$ (resp. $e''$) that are not local extrema of $u|_{e'}$ (resp. $u|_{e''}$). Let $t'=n'-s'$ and $t''=n''-s''$. By Lemma \ref{deg1Mixed}, $v'$ is a degree one vertex of $\Zcal(L_{e'}u)$, and $v''$ is a degree one vertex of $\Zcal(L_{e''}u)$. Since $a_1\neq 0$, $v$ is not a vertex of $\Zcal(L_{e'}u)$ or $\Zcal(L_{e''}u)$. By the same arguments used in the proof of Proposition \ref{acuteNeumann}, each critical point on $e'$ is a vertex of an arc in $\Zcal(L_{e'}u)$ that has a degree one vertex in $e''$ or at $v'$. By Lemma \ref{indexzero}, each degree one vertex of $\Zcal(L_{e'}u)$ (resp. $\Zcal(L_{e''}u)$) on $e''$ (resp. $e'$) is a local extremum of $u|_{e''}$ (resp. $u|_{e'}$). By Remark \ref{atleasttwo}, we find that $2s'+t'-1\leq t''$. Similarly, $2s''+t''-1\leq t'$. It follows that $s'+s''\leq 1$ and $|n'-n''|\leq 1$.\\
    \indent Let $p$ be the nearest local maximum of $u|_{e''}$ to $v$. If $u|_{e''}$ has a local minimum between $p$ and $v$, let $q$ denote this local minimum. We claim that if $q$ exists, then no two arcs in $\Zcal(L_{e'}u)$ with endpoints in $e'$ can have endpoints at both $p$ and $q$. We also claim that if $q$ does not exist, then no arc in $\Zcal(L_{e'}u)$ with an endpoint in $e'$ has an endpoint at $p$. We prove the second claim, and the first claim is proved similarly. Indeed, suppose that $p$ is joined by an arc in $\Zcal(L_{e'}u)$ to $e'$. Then $u(v)<u(p)$ since $u|_{e''}$ has no local minimum between $p$ and $v$. Let $\Omega$ be the region bounded by this arc whose closure contains $v$ (see Figure \ref{obtusenodal}). By Lemma \ref{intFormula} 
    $$\int_{\partial\Omega\cap\partial P}L_{e'}u\partial_{\nu}L_{e'}u<0,$$ contradicting Lemma \ref{negativeintegral}. \\
    \begin{figure}
        \centering
        \includegraphics[width=8cm]{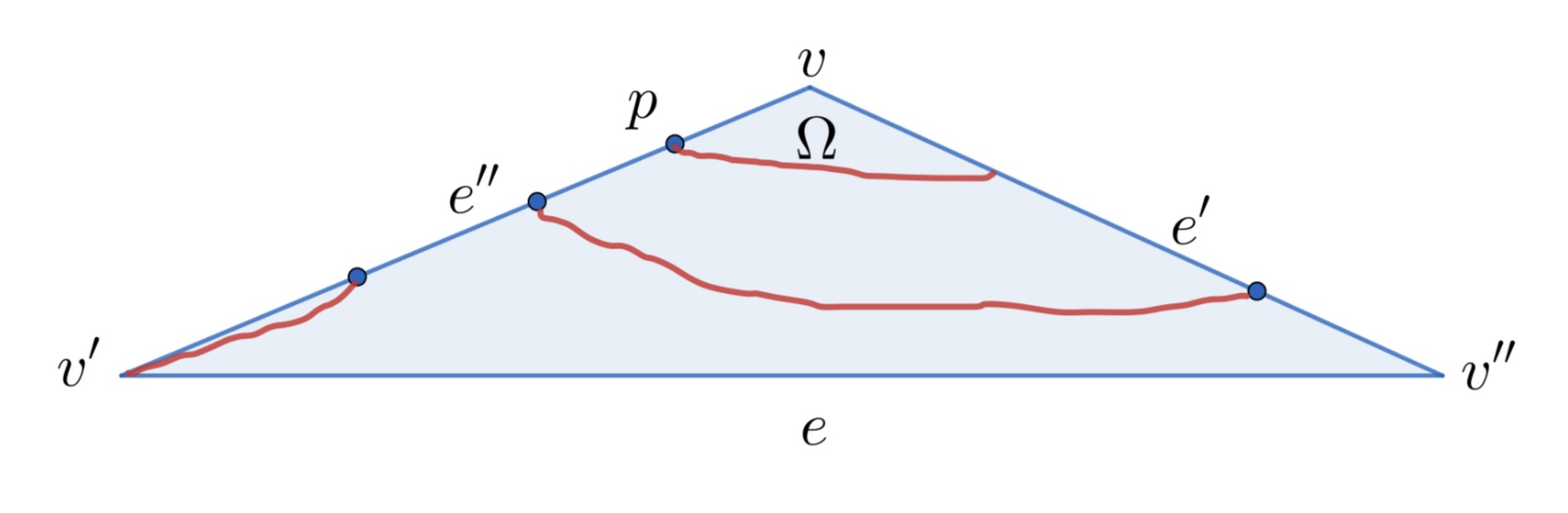}
        \caption{Illustration of the region $\Omega$ constructed in the proof of Proposition \ref{obtuseNeumann}. Here $p$ is the nearest local maximum of $u|_{e''}$ to $v$, and the red arcs represent $\Zcal(L_{e'}u)$.}
        \label{obtusenodal}
    \end{figure}
    \indent We claim that $|n'-n''|=1$. If not then $n'=n''$. Since $u(v')=u(v'')=0$ and $u\geq 0$, the restriction $u|_{e'\cup e''}$ has an odd number of local extrema. Since $n'+n''$ is even, it follows that exactly one of the critical points on the boundary is not a local extremum of $u|_{e'\cup e''}$. Suppose without loss of generality that this critical point lies in $e'$. It follows from Remark \ref{atleasttwo} that $\Zcal(L_{e'}u)$ contains at least $n'+1$ arcs with endpoints in $e'$. Every critical point on $e''$ is then a degree one vertex of one of these arcs. By the previous paragraph, $u|_{e''}$ does not have any local maxima. If $u|_{e''}$ has a local minimum, then there exists a local maximum between this point and $v'$. Thus, $u|_{e''}$ has no local extrema. By the assumption that the only non-extremal critical point lies in $e'$, it follows that $n''=0$. Since $n'=n''=0$, there are no critical points on $e'\cup e''$, contradicting the extreme value theorem. \\
    \indent Therefore, $|n'-n''|=1$. Suppose without loss of generality that $n'=n''+1$. Since $n'+n''$ is odd and there is an odd total number of local extrema of $u|_{e'\cup e''}$, it follows that every critical point on $e'\cup e''$ is a local extremum of $u|_{e'\cup e''}$. Since $n'>n''$, every critical point on $e''$ is a vertex of an arc in $\Zcal(L_{e'}u)$ whose other endpoint lies in $e'$. Since none of these points can be the local maximum of $u|_{e''}$ nearest to $v$, it follows that $u|_{e'}$ has no local maxima and therefore no local minima. Hence, $n''=0$, and $n'=1$, so there are also no interior critical points. Since $u$ must have some global maximum, the critical point $p\in e'$ is the global maximum of $u$. \\
    \indent Now suppose that the unique critical point of $u$ lies in the shorter Neumann edge $e''$. Define the kite $K$ and function $w$ as above. We again have that $w$ does not vanish on the interiors of the edges of $K$ adjacent to $v_1$. Since the critical point of $u$ on $e''$ is the unique local maximum, we have that $w>0$ in a neighborhood of the critical point of $u$. Since $u\geq 0$ and $u(v')=0$, we have that $w<0$ in a neighborhood of $v'$. Thus, there exists an arc in $\Zcal(w)$ with an endpoint in $e''$. As above, $w\equiv 0$ in the line segment bisecting the angle at $v$. The arc with an endpoint in $e''$ cannot intersect this line segment by Lemma \ref{neumannloop}. Since $w\neq 0$ on the interiors of the edges adjacent to $v_1$, the other endpoint of this arc cannot be in one of these edges. Its other endpoint also cannot be in $e''$ by Lemma \ref{neumannloop}, a contradiction.\\
    \indent Whether or not $P$ is isosceles, let $L$ be the constant vector field extending the outward normal vector field to the longer Neumann edge $e'$. Then $Lu|_{e'}\equiv 0$, and $Lu$ cannot vanish in the interiors of the other two edges. By Lemma \ref{deg1Mixed}, the mixed vertex $v'$ opposite to $e'$ is not a vertex of $\Zcal(Lu)$. Thus, if $Lu$ vanishes in $P$, then $\Zcal(Lu)$ must contain a loop, contradicting Lemma \ref{noloops}. Since $u\geq 0$ and $Lu\neq 0$ in the interior of $P$, it follows that $Lu>0$ in $P$. 
    \end{proof}

\begin{proof}[Proof of Theorem \ref{mainthm}]
    If $D$ comprises two edges of $P$, then the result holds by Proposition \ref{acutetwosided}. If $D$ equals a single edge of $P$, then Propositions \ref{acuteNeumann} and \ref{obtuseNeumann} combine to give the result. 
\end{proof}



\end{document}